\documentclass[11pt, a4paper]{amsart}
\usepackage{amsmath,amscd,amssymb, mathtools, enumerate, hyperref}

\usepackage[mathscr]{eucal}

\newtheorem{thm}{Theorem}[section]
\newtheorem{cor}[thm]{Corollary}

\newtheorem{prop}[thm]{Proposition}

\theoremstyle{definition}

\newtheorem{exas}[thm]{Examples}
\newtheorem{exa}[thm]{Example}
\theoremstyle{remark}

\newtheorem{rems}[thm]{Remarks}

\numberwithin{equation}{section}

\def\R{\mathbb{R}}
\def\C{\mathbb{C}}
\def\N{\mathbb{Z}_+}
\def\Z{\mathbb{Z}}
\def\D{\mathbb{D}}
\def\T{\mathbb{T}}

\def\ep{\varepsilon}
\def\l{\lambda}

\def\dd{\,d}
\def\L{\L}
\def\B{\mathcal{B}}
\def\Bes{\mathcal{B}}

\def\Bq{{\Bes_0}}
\def\A{\mathcal{A}}
\def\rit{re^{i\theta}}
\def\E{\mathcal{E}}

\def\L{\mathcal{L}}

\def\W{\mathcal{W}}

\begin{document}

\title[A Katznelson--Tzafriri theorem for analytic Besov functions]{A Katznelson--Tzafriri theorem for analytic Besov functions of operators}

\author{Charles Batty}
\address{St.\ John's College\\
University of Oxford\\
Oxford, OX1 3JP, UK
}

\email{charles.batty@sjc.ox.ac.uk}

\author{David Seifert}
\address{School of Mathematics, Statistics and Physics\\
 Newcastle University\\
 Newcastle upon Tyne, NE1 7RU, UK
}

\email{david.seifert@ncl.ac.uk}

\begin{abstract}
Let $T$ be a power-bounded operator on a Banach space $X$, $\A$ be a Banach algebra of bounded holomorphic functions on the unit disc $\D$, and assume that there is a bounded functional calculus for the operator $T$, so there is a bounded algebra homomorphism mapping functions $f \in \A$ to bounded operators $f(T)$ on $X$.  Theorems of Katznelson--Tzafriri type establish that  $\lim_{n\to\infty} \|T^n f(T)\| = 0$ for functions $f \in \A$ whose boundary functions vanish on the unitary spectrum $\sigma(T)\cap \T$ of $T$, or sometimes satisfy a stronger assumption of spectral synthesis.  We consider the case when $\A$ is the Banach algebra $\B(\D)$ of analytic Besov functions on $\D$.  We prove a Katznelson--Tzafriri theorem for the $\B(\D)$-calculus which extends several previous results.
\end{abstract}



\date \today

\maketitle

\section{Introduction}

Let $T$ be a power-bounded operator on a complex Banach space $X$, $\D$ be the open unit disc, $\T$ be the unit circle, and $\overline\D := \D \cup \T$.    Let $f \in H^\infty(\D)$, the Banach algebra of bounded holomorphic functions on $\D$ with the supremum norm $\|\cdot\|_\infty$.  Then $f$ has a boundary function defined a.e.\ on the unit circle $\T$ as the radial (or non-tangential) limit of $f$  wherever this exists. 
In certain circumstances it may be possible to define a bounded operator $f(T)$ on $X$.   In particular this is possible if $f$ is in the Banach algebra $W^+(\D)$ of all  functions  of the form $f(z) = \sum_{n=0}^\infty a_n z^n$ for $z \in \overline{\D}$, where $\|f\|_W := \sum_{n=0}^\infty |a_n| < \infty$, and $f(T)$ is defined to be $\sum_{n=0}^\infty a_n T^n$.     Then the map $f \mapsto f(T)$ is a bounded algebra homomorphism from $W^+(\D)$ to the Banach algebra $L(X)$ of all bounded linear operators on $X$, mapping the constant function $1$ to the identity operator $I$ and the identity function $z$ to $T$, that is, a bounded $W^+(\D)$-calculus for $T$.   The boundary functions of the functions in $W^+(\D)$ form a closed subalgebra of  the Wiener algebra $W(\T)$  of all functions on $\T$ of the form $f(z) = \sum_{n=-\infty}^\infty a_n z^n$, where $\|f\|_W := \sum_{n=-\infty}^\infty |a_n| < \infty$.

One may hope to define $f(T)$ for all $f$ in the disc algebra $A(\D)$ of all continuous functions on $\overline{\D}$ which are holomorphic on $\D$.   However this is possible if and only if $T$ is polynomially bounded, that is, there is a constant $C$ such that 
\begin{equation*} \label{dpdef}
\|p(T)\| \le C \|p\|_\infty
\end{equation*}
for all polynomials $p$.   When this holds, the map $p \mapsto p(T)$ extends to $A(\D)$ by continuity, since the polynomials are dense in $A(\D)$, and this map is then an $A(\D)$-calculus for $T$.

The term ``Katznelson--Tzafriri theorem'' may be applied to various results about convergence to zero of sequences of the form $(T^nf(T))_{n\ge1}$, where $T$ is a power-bounded operator, $f \in H^\infty(\D)$  and (the boundary function of) $f$ vanishes on the unitary spectrum $\sigma_u(T) := \sigma(T) \cap \T$ of $T$.   In most cases, $f \in A(\D)$ and then $\sigma_u(T)$ is of measure  zero, unless $f$ is the zero function.   The fundamental version, proved by Katznelson and Tzafriri in~\cite[Theorem~5]{KT86}, is as follows.    

 \begin{thm}[Katznelson, Tzafriri] \label{KTthm}
Let $T$ be a power-bounded operator on a Banach space, let $f \in W^+(\D)$, and assume that $f$ is of spectral synthesis in $W(\T)$ with respect to $\sigma_u(T)$.   Then $\lim_{n\to\infty} \|T^n f(T)\| = 0$.
\end{thm}

In Theorem \ref{KTthm} the assumption of spectral synthesis means that there is a sequence $(g_n)_{n\ge1}$ of functions in $W(\T)$ such that each $g_n$ vanishes on a neighbourhood of $\sigma_u(T)$ in $\T$ and $\lim_{n\to\infty} \|g_n-f\|_W = 0$ (here $f$ denotes the boundary function of $f$ on $\T$).   In general, this is stronger than the condition that $f$ vanishes on $\sigma_u(T)$, which is a necessary condition for the conclusion of Theorem \ref{KTthm}; see Remark \ref{remend}(a).   Moreover, if $f \in W^+(\D)$, $f$ vanishes on a closed subset $E$ of $\T$, and $\lim_{n\to\infty} \|T^nf(T)\| = 0$ for every contraction $T$ on an arbitrary Banach space such that $\sigma_u(T) \subseteq E$, then $f$ is of spectral synthesis in $W(\T)$ with respect to $E$;  see \cite[p.\,319]{KT86}, where this fact is attributed to J. Bourgain.

It is desirable to have results in which spectral synthesis is not assumed.   If $\sigma_u(T)$ is countable, and $f$ vanishes on $\sigma_u(T)$, then the condition of spectral synthesis in $W(\T)$ is satisfied \cite[p.\,60]{Kah70}.  Thus, in Theorem \ref{KTthm}, if $\sigma_u(T)$ is countable, the assumption of spectral synthesis can be replaced by the  assumption that $f$ vanishes on $\sigma_u(T)$.

In \cite[Theorem 5]{AOR87}, there is a variant of Theorem \ref{KTthm} in which the assumption of spectral synthesis is replaced by the assumption that $f'\in W^+(\D)$ and $f$ vanishes on $\sigma_u(T)$.   Other related results can be found in \cite{AR89}.

The case when $\sigma_u(T) \subseteq \{1\}$ and $f(z)=1-z$ is a particularly important corollary of Theorem \ref{KTthm}, and it was stated in\cite[Theorem 1]{KT86}.  

\begin{cor}   \label{KTcor}
Let $T$ be a power-bounded operator on a Banach space.     Then $\lim_{n\to \infty} \|T^n(I-T)\| = 0$ if (and only if)   $\sigma_u(T) \subseteq \{1\}$.
\end{cor}

Recently sharp estimates of the rate of decay in Corollary \ref{KTcor}, in terms of the norm of $\|(\l I - T)^{-1}\|$ for $\l \in \T \setminus \{1\}$, have been obtained in \cite{Sei16} (for arbitrary Banach spaces) and \cite{NS20} (for Hilbert spaces).

The following result from \cite[Proposition 1.6]{KvN97} is another variant of Theorem \ref{KTthm} in which spectral synthesis is not assumed.  It will play an important role in this paper.

\begin{thm}[K\'erchy, van Neerven] \label{KvNthm}
 Let $T$ be a polynomially bounded operator on a Banach space.  Let $f \in A(\D)$, and assume that $f$ vanishes on $\sigma_u(T)$.   Then $\lim_{n\to\infty} \|T^nf(T)\| = 0$.
\end{thm}

In \cite{ESZ1} the authors consider certain “isometric function algebras” $R$ of continuous functions on $\T$ and prove a version of the Katznelson--Tzafriri theorem for an operator $T$ on a Banach space satisfying $\|p(T)\| \le \|p\|_R$ for all polynomials $p$, and for suitable functions $f \in A(\D)$ whose boundary functions are of spectral synthesis in $R$ with respect to $\sigma_u(T)$.  
When $R = W(\T)$, \cite[Theorem 2.10]{ESZ1} gives essentially the same result as Theorem \ref{KTthm}.  When $R = C(\T)$, it recovers Theorem \ref{KvNthm} for contractions on Hilbert spaces, using von Neumann's inequality.    More recently, the following, more general, result for power-bounded operators on Hilbert spaces  was proved in \cite[Theorem~2.1]{Lek09} by considering a certain ergodic condition.

\begin{thm}[L\'eka] \label{Lekthm}
Let $T$ be a power-bounded operator on a Hilbert space.  Let $f \in W^+(\D)$, and assume that $f$ vanishes on $\sigma_u(T)$.   Then
$$\lim_{n\to\infty} \|T^nf(T)\| = 0.$$
\end{thm}

 If $T$ is a completely non-unitary contraction on a Hilbert space, $f \in H^\infty(\D)$ and the boundary function of $f$ vanishes on $\sigma_u(T)$, then Bercovici \cite{Ber90b} showed that $\lim_{n\to\infty} \|T^nf(T)\| = 0$, where $f(T)$ is defined in the Sz.-Nagy-Foias functional calculus.

Other results of Katznelson--Tzafriri type in slightly different situations have been proved in \cite{Aba16}, \cite{Ber90} and \cite{Zar13}, for example.   Short surveys of the topic can be found in \cite[Section 5]{CT04} and \cite{Lek17}, and \cite{BS22a} provides a longer survey.

The Banach algebra $\B(\D)$ of analytic Besov functions on $\D$ has been considered in various contexts in the literature.  It is often defined in terms of the boundary functions on $\T$, but it can alternatively be defined directly from functions on $\D$;  see \eqref{bddef} below.   This algebra satisfies $W^+(\D) \subsetneq \B(\D) \subsetneq A(\D)$ with continuous inclusions, and moreover the polynomials are dense in $\B(\D)$.  Nevertheless, $\B(\D)$ does not appear to fit into the framework of~\cite{ESZ1}.

In this paper we consider a given power-bounded operator $T$ on a Banach space $X$ which has a bounded $\B(\D)$-calculus (abbreviated to $\B$-calculus in the rest of this paper), that is, there exists a (unique) bounded algebra homomorphism from $\B(\D)$ to $L(X)$ mapping $1$ to $I$ and $z$ to $T$.   Any polynomially bounded operator has a bounded $\B$-calculus.  
 Peller showed in \cite{Pel82} that a $\B$-calculus exists for all power-bounded operators on a Hilbert space, and Vitse showed in \cite{Vit04} and \cite{Vit05} that a $\B$-calculus exists for all Ritt operators on any Banach space.  
The recent papers \cite{BGT} and \cite{BGT2} have developed a detailed theory of the analogous $\B$-calculus for (negative) generators of bounded $C_0$-semigroups, involving the Banach algebra $\B(\C_+)$, where $\C_+$ is the right half-plane.  
The disc case can be treated more easily, due to the density of the polynomials and the operators being bounded.  Arnold \cite{Arn21} has adapted some results from \cite{BGT} and \cite{BGT2} to characterise the operators $T$ which have a $\B$-calculus in terms of an integral condition.   We were unaware of Arnold's work while we were producing this paper, and we have obtained the same characterisation independently.  Our method is different in various ways, and we present it here in order to make the paper more self-contained and to provide a formulation which is better suited to our purposes, for example providing the spectral mapping theorem in Theorem \ref{SIT}.

We prove the following version of the Katznelson--Tzafriri theorem for operators $T$ which have a bounded $\B$-calculus, and functions $f \in \B(\D)$ which vanish on $\sigma_u(T)$.   Theorem \ref{KTB} contains Theorem~\ref{Lekthm} as a special case, and it applies to many more operators than Theorem \ref{KvNthm} does, although $\B(\D)$ is smaller than $A(\D)$.  It applies to fewer operators than Theorem \ref{KTthm} does, but it extends the class of functions $f$ from $W^+(\D)$ to $\B(\D)$ and it does not require spectral synthesis.

\begin{thm} \label{KTB}
Let $T$ be a power-bounded operator on a Banach space, and assume that $T$ has a bounded $\B$-calculus.   Let $f \in \B(\D)$, and assume that $f$ vanishes on $\sigma_u(T)$.   Then $\lim_{n\to\infty} \|T^nf(T)\| = 0$.
\end{thm}

The remainder of this paper has three sections.   In Section \ref{BD}, we set out several properties of $\B(\D)$, including a reproducing formula for functions in $\B(\D)$ (Proposition \ref{repf}).  This section is broadly analogous to parts of \cite[Section 2]{BGT} and \cite[Section 3]{BGT2}.  
Section~\ref{BC} contains our approach to the $\B$-calculus, and its properties.   
In Section \ref{KT}, we prove Theorem \ref{KTB} and we mention a more general version (Theorem \ref{KTB3}).   The proof uses a combination of ideas from earlier results of this type, together with Theorems \ref{KvNthm} and~\ref{cfbd}, and Proposition \ref{isom} which is extracted from \cite[Section 2]{CG08}.

\section{The algebra $\B(\D)$} \label{BD}

There is more than one way to define the algebra $\B(\D)$ of analytic Besov functions on the unit disc $\D$.   
  For our purposes, it is convenient to define it as the space of all holomorphic functions $f$ on $\D$ such that
\begin{equation} \label{bddef}
\|f\|_{\B_0} := \int_0^1 \sup_{|\theta|\le\pi} |f'(re^{i\theta})| \, dr < \infty.
\end{equation}
Note that $\|\cdot\|_{\B_0}$ is a seminorm on $\B(\D)$.  

This definition of $\B(\D)$ is analogous to the definition of the algebra of analytic Besov functions on $\C_+$ considered in \cite{BGT}, \cite{BGT2} and \cite{V1} (and denoted by $\B(\C_+)$ in this paper).   Alternatively, $\B(\D)$ may be defined as the space of bounded holomophic functions on $\D$ whose boundary functions on $\T$ belong to a standard Besov space often denoted by $B^0_{\infty1}(\T)$, and this is the reason why functions in $\B(\D)$ are called analytic Besov functions.     There are many equivalent norms on $B^0_{\infty1}(\T)$ and they are usually defined via harmonic analysis.   Peller \cite{Pel82} defines the space $\B(\D)$ by that method and denotes it by $B^0_{\infty1}$.  In \cite{Vit05}, Vitse gives both definitions, denoting the space by $B^0_{\infty1}(\D)$, and using the definition by harmonic analysis in the proofs.   

In the following proposition we set out some properties of $\B(\D)$.  We omit the proofs, as they appear to be common knowledge.    They can be proved quite easily from our definition, and the first five statements are analogous to properties of $\B(\C_+)$ established in \cite[Section 2.2]{BGT}.   The fact that the inclusions in part (d) are proper is a well-known property of the spaces of boundary functions, and it is reflected in the existence of operators which have functional calculus with respect to the smaller algebra but not the larger algebra;  see Example \ref{PelVit}(a) and the discussion after Proposition~\ref{isom}.   Property (f) differs considerably from the case of $\B(\C_+)$, but it can easily be deduced from (e) and the Taylor series expansion of $f_r$ in $\B(\D)$.

\begin{prop}\phantomsection\label{prp:B_prop}
\begin{enumerate} [\rm(a)]
\item  If $f \in \B(\D)$, then $f$ is uniformly continuous, so $f$ extends to a function in $A(\D)$.
\item  A norm on $\B(\D)$ is defined by
\[
\|f\|_\B := \|f\|_\infty + \|f\|_{\B_0}, \qquad f \in \B(\D).
\]
\item  $(\B(\D), \|\cdot\|_\B)$ is a Banach algebra.
\item There are continuous proper inclusions:  $W^+(\D) \subsetneq \B(\D) \subsetneq A(\D)$.
\item 
Let $f\in\B(\D)$, and $f_r(z) := f(rz)$ for $r \in (0,1)$, $z\in\D$.  Then $f_r \in W^+(\D)$ and 
\begin{equation*} \label{frf}
\lim_{r \to 1-} \|f_r-f\|_\B = 0.
\end{equation*}
\item  The polynomials are dense in $\B(\D)$.
\end{enumerate}
\end{prop}

The $\B$-norm $\|\cdot\|_\B$ is equivalent to any of the standard norms arising from equivalent definitions of $\B(\D)$.   The Gelfand space of $\B(\D)$ consists of the point evaluations $f \mapsto f(z)$ for each $z \in \overline\D$.

Let $\E(\D)$ be the space of all holomorphic functions $g$ on $\D$ such that
\[
\|g\|_{\E_0} := \sup_{0<r<1} (1-r) \int_{-\pi}^\pi |g'(re^{i\theta})|\,d\theta < \infty.
\]
This space is the direct analogue of the space $\E$ of functions on $\C_+$ considered in \cite[Section 2.5]{BGT}.  
Note that the seminorm $\|\cdot\|_{\E_0}$ is unchanged if the supremum is taken over $r\in(a,1)$ where $a\in(0,1)$ or if $1-r$ is replaced by $r \log r^{-1}$. 

\begin{exas} \label{res}
(a)  For polynomials $p$ of degree $n\ge1$, the $\B$-norm satisfies
\[
\|p\|_\B \le C \log(n+1) \|p\|_\infty,
\]
where $C$ is an absolute constant.  A proof of this, based on harmonic analysis, is given in \cite[Lemma 2.3.7]{White}.  Related estimates for certain classes of operators appear in \cite{Pel82} and \cite{Vit05}.  

\noindent 
(b)  For $w \in \D$, let
\begin{equation*} \label{rhow}
\rho_w(z) := (1 - wz)^{-1},  \qquad z \in \D.
\end{equation*}
Then 
\[
\|\rho_w\|_{\Bq} = \int_0^1 \sup_{|z|=r} \left| \frac{w}{(1-wz)^2} \right|\,dr = \int_0^1 \frac{|w|}{(1-|w|r)^2} \,dr = \frac{|w|}{1-|w|}.
\]
Thus $\rho_w \in \B(\D)$.  Since $\|\rho_w\|_\infty= (1-|w|)^{-1}$,  we obtain 
\begin{equation} \label{Brho}
\|\rho_w\|_\B = \frac{1+|w|}{1-|w|}.
\end{equation}

From standard integration methods, one can obtain, for $t \in (0,1)$,
\begin{equation}  \label{valint}
\int_{-\pi}^\pi \frac{d\theta}{|1-te^{i\theta}|^2 } =
\int_{-\pi}^\pi \frac{d\theta} {1+t^2-2t\cos\theta} = \frac{2\pi}{1-t^2}.
\end{equation}
Hence, for $r \in (0,1)$,
\begin{multline*}
\int_{-\pi}^\pi \big| \rho_w'(re^{i\theta}) \big| \,d\theta =  \int_{-\pi}^\pi \frac{|w|}{|1 - |w|r e^{i\theta}|^2} \,d\theta
= \frac{2\pi|w|}{1-|w|^2r^2} \le \frac{2\pi|w|}{1-r}.
\end{multline*}
Thus $\rho_w \in\E(\D)$ and $\|\rho_w\|_{\E_0} \le 2\pi|w|$.
\end{exas}

There is a partial duality between the spaces $\B(\D)$ and $\E(\D)$.   For $f \in \B(\D)$ and $g \in \E(\D)$, define
\begin{align*}
\langle g,f \rangle_\B &:= \int_0^1 r \log \frac{1}{r} \int_{-\pi}^\pi f'(\rit)g'(re^{-i\theta}) \,d\theta\,dr \\
&= \int_{\D} \log  \frac{1}{|z|} \,  f'(z) g'(\overline{z}) \,dA(z),
\end{align*}
where $dA$ denotes area measure on $\D$.   The following estimates show that the integrals above are absolutely convergent:
\begin{align}\label{pdc}
|\langle g,f \rangle_\B| &\le \int_0^1 r \log \frac{1}{r}\int_{-\pi}^\pi \big| f'(\rit)g'(re^{-i\theta}) \big| \,d\theta\,dr  \\
&\le \int_0^1 \sup_{|\psi|\le\pi} \big|f'(r e^{i\psi})\big| (1-r) \int_{-\pi}^\pi \big|g'(\rit)\big| \,d\theta\,dr \nonumber \\
&\le \|g\|_{\E_0} \|f\|_{\B_0}. \nonumber
\end{align}

In \cite{BGT}, the space $\E$ corresponding to $\E(\D)$ plays a major role in constructing the $\B$-calculus on $\C_+$.   In this paper we use $\E(\D)$ only in the context of \eqref{pdc}.

We now give a reproducing formula for functions in $\B(\D)$, corresponding to \cite[(3.3)]{Arn21} and \cite[Proposition 2.20]{BGT}.

\begin{prop}  \label{repf}
Let $f \in \B(\D)$.   Then
\begin{equation} \label{rep}
f(w) = f(0) + \frac{2}{\pi} \langle \rho_w,f \rangle_{\B}, \qquad w\in\D.
\end{equation}
\end{prop}

\begin{proof}
By density of the polynomials in $\B(\D)$, and linearity and continuity of the partial duality given by \eqref{pdc}, it suffices to establish \eqref{rep} for $f(w) := w^k$, where $k\in\N$.   The case $k=0$ is trivial.   If $k\ge1$, then $f(0)=0$ and
\begin{align*}
\langle \rho_w,f \rangle_{\B} &= \int_\D \log \frac{1}{|z|} \, kz^{k-1} w (1-w\overline{z})^{-2} \, dA(z) \\
&= kw \int_0^1 \int_{-\pi}^\pi \log \frac{1}{r} \, \sum_{n=0}^\infty (n+1) w^n r^{k+n} e^{i(k-1-n)\theta} \,d\theta \, dr.
\end{align*}
For fixed $r \in (0,1)$, the series converges uniformly in $\theta$, and the integral of the $n$th term with respect to $\theta$ is $0$ for each $n$ except $n = k-1$.  Thus
\[
\frac{2}{\pi} \langle \rho_w,f \rangle_{\B} = 4 k^2 w^k \int_0^1 \log  \frac{1}{r} \,r^{2k-1} \,dr = w^k,
\]
as required.
\end{proof}

Let $\W(\D)$ be the space of all continuous functions $g : \D \to \C$ such that
\[
\|g\|_\W :=\int_0^1 \sup_{|\theta|\le \pi} |g(\rit)| \, dr < \infty.
\]
For $g \in \W(\D)$, define
\[
(Qg)(w) := \frac{2}{\pi} \int_\D \log  \frac{1}{|z|} \frac {wg(z)}{(1-w\overline{z})^2} \, dA(z),  \qquad w \in \D.
\]
This integral is absolutely convergent, by essentially the same estimation as in \eqref{pdc} (with $g$ and $f'$ replaced by $\rho_w$ and $g$, respectively).   These definitions and Proposition \ref{dgrh} below are analogues of \cite[Proposition 3.1]{BGT2}.

If $f \in \B(\D)$, then $f' \in \W(\D)$ and \eqref{rep} may be rewritten as
\[
f = f(0) + Q(f').
\]
Thus $Q$ maps the space of holomorphic functions in $\W(\D)$ onto the space of functions $f \in \B(\D)$ with $f(0)=0$.  Our next result shows that $Q$ maps the whole space $\W(\D)$ into $\B(\D)$.

For $r \in (0,1)$ and $h \in L^p(-\pi,\pi)$, where $1 \le p \le \infty$, let
\begin{equation} \label{grh}
G_{r,h}(w) := \int_{-\pi}^{\pi} \frac {h(\theta)}{(1-wre^{-i\theta})^2} \, d\theta, \qquad w \in \D.
\end{equation}
If $g\in\W(\D)$ and $g_r(\theta):=g(re^{i\theta})$ for $r\in(0,1)$, $\theta\in[-\pi,\pi]$, then
\begin{equation} \label{eq:QG}
(Qg)(w)=\frac{2}{\pi} \int_0^1 r\log \frac{1}{r} \,w G_{r,g_r}(w) \, dr,  \qquad w \in \D.
\end{equation}
The functions $G_{r,h}$ correspond to the functions in \cite[Proposition 3.1]{BGT2}.  They belong to $\B(\D)$ and they will play an important role in Section~\ref{BC}.

\begin{prop} \label{dgrh}
\begin{enumerate}[\rm(a)]
\item  Let $r \in (0,1)$ and $h \in L^1(-\pi,\pi)$.  Then  $G_{r,h} \in \B(\D)$ and
\begin{equation} \label{B09}
G_{r,h} = \int_{-\pi}^\pi h(\theta)\rho_{r e^{-i\theta}}^2 \,d\theta,
\end{equation}
where the right-hand side exists as a $\B(\D)$-valued Bochner integral.
\item  Let $r \in (0,1)$ and $h\in L^\infty(-\pi,\pi)$.   Then $G_{r,h} \in \B(\D)$ and
\begin{equation}\label{B11}
\|G_{r,h}\|_{\Bes}\le \frac{6\pi}{1-r} \|h\|_\infty.
\end{equation}
\item $Q$ is a bounded linear operator from $\W(\D)$ to $\B(\D)$.
\end{enumerate}
\end{prop}

\begin{proof}
(a)  From \eqref{Brho}, the map $w \mapsto \rho_w$ is locally bounded from $\D$ to the Banach algebra $\B(\D)$. For each $z\in \D$, the point evaluation $f \mapsto f(z)$ is a bounded linear functional on $\B(\D)$, and the function $w \mapsto \rho_w(z)^2$ is holomorphic on $\D$.  By \cite[Corollary A.7]{ABHN} the function $w\mapsto \rho_w^2$ is holomorphic as a map from $\D$ to $\B(\D)$.   Moreover, $\|\rho_{r e^{-i\theta}}^2\|_\B \le 4(1-r)^{-2}$ by~\eqref{Brho}.    For $h \in L^1(-\pi,\pi)$, it follows that
\[
J_{r,h} := \int_{-\pi}^\pi h(\theta) \rho_{r e^{-i\theta}}^2 \, d\theta
\]
exists as a $\B(\D)$-valued Bochner integral, and
\[
J_{r,h}(w) = \int_{-\pi}^\pi h(\theta) \rho_{r e^{-i\theta}}(w)^2 \, d\theta = G_{r,h}(w),\qquad w\in\D,
\]
as required.

\noindent
(b)  For $h \in L^\infty(-\pi,\pi)$, it follows from (a) that $G_{r,h}$ is holomorphic on $\D$. For $w\in\D$,
\[
|G'_{r,h}(w)|  = \bigg|2r \int_{-\pi}^\pi \frac{h(\theta) e^{-i\theta}}{(1-wr e^{-i\theta})^3} \,d\theta\bigg|\le 2r \|h\|_\infty  \int_{-\pi}^\pi \frac{d\theta}{|1-wre^{-i\theta}|^3}.
\]
If $w \in \D$ and  $t := |w|r$, then 
\begin{equation} \label{intest}
\begin{aligned}
 \int_{-\pi}^\pi \frac{d\theta}{|1-wre^{-i\theta}|^3} 
 &=  2 \int_0^\pi \frac{d\theta}{(1+t^2-2t \cos\theta)^{3/2}} \\
 &= 4 \int_\R \frac{(1+u^2)^{1/2}} {\left((1-t)^2 + (1+t)^2u^2\right)^{3/2}}  \, du  \\
 &= \frac{4}{(1-t)^2(1+t)} \int_0^\infty  {\bigg( 1+ \left(\frac{1-t}{1+t}\right)^2  v^2\bigg)^{1/2} } \frac{dv}{(1+v^2)^{3/2}} \\
 &\le \frac{4}{(1-t)^2} \int_0^\infty \frac{dv}{1+v^2} = \frac{2\pi}{(1-t)^2}, 
\end{aligned}\end{equation}
where we have made the substitutions $u = \tan (\theta/2)$ and $v = (1+t)(1-t)^{-1}u$.   
 It follows that 
\begin{equation}\label{eq:G_B0}
\|G_{r,h}\|_\Bq \le 4 \pi  r \|h\|_\infty \int_0^1 \frac{ds}{(1-sr)^2} = \frac{4 \pi r}{1-r} \|h\|_\infty\le  \frac{4 \pi }{1-r} \|h\|_\infty.
\end{equation}
From \eqref{valint}, or by using the same substitutions as above, we see that
\begin{equation}\label{eq:G_inf}
\|G_{r,h}\|_\infty \le \sup_{|w|<1} \int_{-\pi}^\pi \frac{\|h\|_\infty}{\left|1-wre^{-i\theta}\right|^2} \, d\theta =  \frac {2\pi}{1-r^2}\|h\|_\infty  \le \frac{2\pi}{1-r} \|h\|_\infty,
\end{equation}
and now \eqref{B11} follows from \eqref{eq:G_B0} and \eqref{eq:G_inf}.

\noindent
(c) Let $g\in\W(\D)$.  It is not difficult to see, for instance using Morera's theorem, that $Qg$ is holomorphic on $\D$,  and  $Q$ is certainly linear. Let  $w\in\D$ and  $g_r(\theta):=g(re^{i\theta})$ for $r\in(0,1)$, $\theta\in[-\pi,\pi].$   Using $r\log r^{-1}\le 1-r$ for $r\in(0,1)$, \eqref{eq:QG} and \eqref{eq:G_inf}, we see that
\begin{equation}\label{eq:G_est}
\begin{aligned}
|(Qg)(w)|&=\frac{2|w|}{\pi}\left|\int_0^1 r\log\frac1r\, G_{r,g_r}(w) \dd r\right|\\&\le4\int_0^1\frac{r\log r^{-1}}{1-r}\|g_r\|_\infty\, dr\le4\|g\|_\W,
\end{aligned}
\end{equation}
and hence $\|Qg\|_\infty\le4\|g\|_\W$.

Next we note that
$$(Qg)'(w)=\frac{2}{\pi}\int_0^1 r\log\frac1r\left( G_{r,g_r}(w)+\int_{-\pi}^\pi\frac{2wre^{-i\theta}g(re^{i\theta})}{(1-wre^{-i\theta})^3}\dd\theta\right)d r.
$$
Using  the same estimate as in \eqref{eq:G_est}, and \eqref{intest}, we obtain
$$|(Qg)'(w)|\le4\|g\|_\W+8\int_0^1\frac{r^2\log r^{-1}}{(1-|w|r)^2}\|g_r\|_\infty\dd r,$$
and, using Fubini's theorem, it follows that
$$\begin{aligned}
\|Qg\|_{\B_0}&\le 4\|g\|_\W+8\int_0^1r^2\log\frac1r\,\|g_r\|_\infty\int_0^1\frac{d s}{(1-sr)^2}\dd r\\&
\le 4\|g\|_\W+8\int_0^1\frac{r\log r^{-1}}{1-r}\|g_r\|_\infty\dd r
\le 12\|g\|_\W.
\end{aligned}$$
Hence $Qg\in\B(\D)$ and $\|Qg\|_\B\le 16\|g\|_\W$, so $Q$ maps $\W(\D)$ boundedly into $\B(\D)$, as required.
\end{proof}

\section{The $\B$-calculus}  \label{BC}

Let $T$ be a bounded operator on a complex Banach space $X$ with spectral radius at most $1$.   In this section we show that a resolvent condition on $T$ is equivalent to $T$ having a bounded $\B$-calculus.   It is a discrete analogue of a condition on unbounded operators which was introduced by Gomilko \cite{Gom99} and independently by Shi and Feng \cite{SF00}, and which is equivalent to the existence of a bounded $\B(\C_+)$-calculus for some classes of unbounded operators.    

For $x \in X$ and $x^* \in X^*$, we define
\[
g_{x,x^*}(z) := \langle (I-zT)^{-1}x,x^* \rangle, \qquad z \in \D.
\]
We say that $T$ satisfies the {\it discrete Gomilko-Shi-Feng condition} if $g_{x,x^*} \in \E(\D)$ for all $x\in X$ and $x^*\in X^*$,  that is,
\begin{equation} \label{dgsf}
\sup_{0<r<1} (1-r) \int_{-\pi}^\pi \big| \langle T(I-\rit T)^{-2}x,x^* \rangle \big| \,d\theta < \infty.
\end{equation}
As noted in Section \ref{BD} we may take the supremum over $r\in(a,1)$ where $a \in (0,1)$, instead of over $r\in(0,1)$, or we may replace the factor $1-r$ by $r \log r^{-1}$.   In addition, the Closed Graph Theorem implies that there exists a finite constant $\gamma_T$ such that the supremum in \eqref{dgsf} is bounded by $\gamma_T \|x\| \,\|x^*\|$ for all $x \in X$ and $x^* \in X^*$.

It is easy to see that the condition \eqref{dgsf} is equivalent to property (GFS)$_1$ in \cite[Definition 2.2]{Arn21} and to the property (2.5) in \cite{CG08}.     We shall call any of these equivalent properties the {(dGSF) condition}.   It was shown in \cite[Lemma 2.1]{CG08} that the (dGSF) condition implies that $T$ is power-bounded. 

We now obtain some other variations of \eqref{dgsf}.

\begin{prop}  \label{vdgsf}
Let $T$ be a power-bounded operator on a complex Banach space $X$. The following are equivalent:
\begin{enumerate} [\rm(i)]
\item $T$ satisfies \eqref{dgsf} for all $x\in X$ and $x^*\in X^*$.
\item For all $x\in X$ and $x^*\in X^*$,
\begin{equation} \label{dgsf2}
\sup_{0<r<1} (1-r) \int_{-\pi}^\pi \big| \langle (I-\rit T)^{-2}x,x^* \rangle \big| \,d\theta < \infty.
\end{equation}
\item  For all $x \in X$ and $x^* \in X^*$,
\begin{equation} \label{dgsf3}
\sup_{r>1} \, (r-1) \int_{-\pi}^\pi \big| \langle (\rit I- T)^{-2}x,x^* \rangle \big| \,d\theta < \infty.
\end{equation}
\end{enumerate}
\end{prop} 

\begin{proof}
By passing to an equivalent norm on $X$, we may assume that $\|T\|\le1$, so that $\|(I - \rit T)^{-1} \| \le (1-r)^{-1}$ for $r \in (0,1)$.

Let  $x \in X$, $x^* \in X^*$ and $r \in (0,1)$.   Then
\begin{align*}
\lefteqn {\left| \int_{-\pi}^\pi \big| \langle (I-\rit T)^{-2}x,x^* \rangle \big| \,d\theta -
 \int_{-\pi}^\pi \big| \langle  T (I-\rit T)^{-2}x,x^* \rangle \big| \,d\theta  \right|} \\
 &\le \int_{-\pi}^\pi \big| \big\langle \big( (1-\rit T)^{-2} -  \rit T (1-\rit T)^{-2} \big) x,x^* \big\rangle \big| \,d\theta \\
&\null\hskip20pt
 +  \int_{-\pi}^\pi\big| \langle (1-r) e^{i\theta} T (1 - \rit T)^{-2}x,x^* \rangle  \big| \,d\theta \\
 &\le \int_{-\pi}^\pi \big| \langle (I-\rit T)^{-1}x, x^* \rangle \big| \, d\theta  + (1-r) \int_{-\pi}^\pi\big| \langle  T (I - \rit T)^{-2}x,x^* \rangle  \big| \,d\theta\\
& \le 2\int_{-\pi}^\pi (1-r)^{-1} \|x\|\,\|x^*\| \, d\theta  = \frac{4\pi}{1-r} \|x\|\,\|x^*\|.
\end{align*}
The equivalence of (i) and (ii) follows immediately from this.

The equivalence of (ii) and (iii) follows easily on replacing $r$ by $r^{-1}$ and $\theta$ by $-\theta$ in the integrals.
\end{proof}

Now assume that $T$ satisfies \eqref{dgsf}.  Let $f \in \B(\D)$,  and define
\begin{equation} \label{dBfc}
\langle f(T)x, x^* \rangle := f(0) \langle x,x^* \rangle +\frac2\pi \langle g_{x,x^*}, f \rangle_\B, \qquad x\in X, \, x^*\in X^*.
\end{equation}
This formula is of the same format as \cite[(3.4)]{Arn21}, although the kernels in the double integrals are different.   The following theorem is essentially the same as \cite[Theorem 3.2]{Arn21}.

\begin{thm} \label{dbfc}
Let $T$ be a power-bounded operator on a Banach space $X$.   If $T$ satisfies the {\rm (dGSF)} condition, then $f \mapsto f(T)$ is a $\B$-calculus for the operator $T$.   Moreover,
\begin{equation}  \label{gtbd}
\|f(T)\| \le \max \left\{1,\frac{2\gamma_T}{\pi}\right\} \|f\|_\B, \qquad f \in \B(\D).
\end{equation}
\end{thm}

\begin{proof}
It is clear from \eqref{pdc} and the finiteness of $\gamma_T$ that $f(T)$ is a bounded linear operator from $X$ into $X^{**}$ for every $f\in\B(\D)$.   We need to show that $f(T)$ maps $X$ into $X$, and $f \mapsto f(T)$ is an algebra homomorphism mapping $1$ to $I$ and $z$ to $T$.    By the density of the polynomials in $\B(\D)$ and linearity, it suffices to show that if $f(z)=z^k$ for some $k\in\N$, then  $f(T) = T^k$.    The proof is almost identical to that of Proposition \ref{repf}, with $w$ replaced by $T$, and $x$ and $x^*$ inserted. 

 The case $k=0$ is trivial.   For $k\ge1$, we have $f(0)=0$.  For $x \in X$ and $x^* \in X^*$,
\begin{align*}
\langle f(T)x&,x^*\rangle = \frac{2}{\pi} \int_\D \log \frac{1}{|z|} \, kz^{k-1} \langle T (I-\overline{z}T)^{-2}x, x^* \rangle  \, dA(z) \\
&= \frac{2k}{\pi} \int_0^1 \int_{-\pi}^\pi \log \frac{1}{r} \,  \sum_{n=0}^\infty (n+1) \langle T^{n+1}x, x^* \rangle r^{k+n} e^{i(k-1-n)\theta} \,d\theta \, dr.
\end{align*}
For fixed $r \in (0,1)$, the series converges uniformly in $\theta$, and the integral of the $n$th term with respect to $\theta$ is $0$ for each $n$ except $n = k-1$.  Thus
\[
\langle f(T)x, x^*  \rangle = 4 k^2 \langle T^k x,x^* \rangle \int_0^1 \log \frac{1}{r} \, r^{2k-1} \,dr = \langle T^k x,x^* \rangle.
\]
It follows that the map $f \mapsto f(T)$ is an algebra homomorphism from the polynomials into $L(X)$, so the map $f \mapsto f(T)$ is a $\B$-calculus for  $T$.  The estimate \eqref{gtbd} follows from \eqref{dBfc}, \eqref{pdc}, and the definition of $\gamma_T$.
\end{proof}

The $\B$-calculus defined by \eqref{dBfc} is the \emph{unique} $\B$-calculus for $T$, since the the polynomials are dense in $\B(\D)$.   The $\B$-calculus can also be described in the following ways, where we write $r(T)$ for the spectral radius of $T$.

\begin{cor} \label{frt}
Let $T$ be a power-bounded operator satisfying the {\rm(dGSF)} condition, and let $f$ be a holomorphic function on $\D$ with Taylor series $f(z) = \sum_{n=0}^\infty a_nz^n$.   Then the following statements hold in the operator-norm topology:
\begin{enumerate}[\rm(a)]
\item  If $f \in \B(\D)$, then $f(T) = \lim_{r \to 1-} \sum_{n=0}^\infty a_n r^n T^n.$
\item  If $f \in \B(\D)$ and $r(T)<1$, then $f(T) = \sum_{n=0}^\infty a_n T^n$.
\end{enumerate}
\end{cor}

\begin{proof}
Let $f\in\B(\D)$ and $f_r(z):=f(rz)$ for $r\in(0,1)$, $z\in\D$. By Proposition~\ref{prp:B_prop}(e), $f_r\in W^+(\D)$ and $\lim_{r\to1-}\|f_r-f\|_\B=0$. Now $f_r(T)=\sum_{n=0}^\infty a_n r^n T^n$ for every $r\in(0,1)$ and, since $T$ has a bounded $\B$-calculus by Theorem~\ref{dbfc}, $f(T)=\lim_{r\to1-}f_r(T)$ in operator norm. Thus (a) follows, which in turn implies~(b).
\end{proof}

Since $\|\cdot\|_\B$ is dominated by $\|\cdot\|_W$ on $W^+(\D)$, the estimate \eqref{gtbd} is generally sharper than the estimate $\|f(T)\| \le \sup_{n\in\Z_+} \|T^n\|\, \|f\|_W$ when $f \in W^+(\D)$.

\begin{exas}  \label{PelVit}
(a)   Let $T$ be a power-bounded operator on a Hilbert space.   It is shown in \cite[Corollary 2.5]{CG08}, and also in \cite[Proposition 6.4.4]{Sei14} and \cite[Corollary 2.13]{Arn21}, that \eqref{dgsf3} holds.   The argument uses Parseval's theorem.   
It follows from Proposition~\ref{vdgsf} and Theorem \ref{dbfc} that $T$ has a bounded $\B$-calculus.  As stated in the introduction, this fact follows from estimates of Peller \cite{Pel82} for polynomial functions of $T$, but no explicit formula for $f(T)$ is given there.

There exist power-bounded operators on Hilbert spaces which are not polynomially bounded; see \cite{Leb68}, for example. Such operators have a bounded $\B$-calculus, but they do not have a bounded $A(\D)$-calculus.    This confirms that $\B(\D)$ is properly included in $A(\D)$. 

\noindent 
(b)  Let $T$ be in the class of Ritt operators, sometimes known as Tadmor-Ritt operators.   This means that $T$ is a bounded linear operator on a Banach space,  $r(T)\le 1$ and there is a constant $C$ such that
\[
\|(zI-T)^{-1}\| \le C |z-1|^{-1}, \qquad |z|>1.
\]
It was proved in \cite{Kom68} and \cite[Theorem~4.5.4]{Nev93} that $T$ is power-bounded; see also~\cite{Bak88}.   Further results considering polynomial boundedness and functional calculus on Stolz domains contained in $\D$ were obtained in \cite{LL15} and \cite{LeM14}.  In various respects Ritt operators are analogues of generators of bounded holomorphic $C_0$-semigroups.

For $r>1$, we have
\[
\int_{-\pi}^\pi \|(re^{i\theta}I - T)^{-2}\| \, d\theta \le \frac{C}{r^2} \int_{-\pi}^\pi \frac{ d\theta}{|1 - r^{-1}e^{-i\theta}|^2} = \frac{2\pi C}{r^2 (1-r^{-1})} \le \frac{2\pi C}{r-1}, 
\]
and hence  \eqref{dgsf3} holds.  
It follows from Proposition~\ref{vdgsf} and Theorem \ref{dbfc} that $T$ has a bounded $\B$-calculus.  In this case, the definition of $f(T)$ in \eqref{dBfc} can be strengthened to the operator-valued formula
\begin{equation} \label{ritt}
f(T) = f(0)I + \frac{2}{\pi} \int_\D \log \frac{1}{|z|}\, f'(z) T (I- \overline{z}T)^{-2} \, dA(z).
\end{equation}
As stated in the introduction, Vitse proved the existence of a bounded $\B$-calculus for Ritt operators in \cite{Vit05}, based on estimates for polynomial functions of $T$, but no explicit formula for $f(T)$ is given in that paper.
\end{exas}

Next we will obtain a converse of Theorem \ref{dbfc}, by showing that if $T$ has a bounded $\B$-calculus, then (dGSF) holds.    The argument for this is similar to the continuous-parameter case \cite[Theorem 6.1]{BGT2}, but it is more direct than the proof of the same result in \cite[Theorem 3.5]{Arn21}.   The result will be used in the proof of Theorem \ref{KTB}.

\begin{thm} \label{cfbd}
Let $T$ be a bounded operator on a  Banach space $X$.   If $T$ admits a bounded $\B$-calculus, then $T$ is power-bounded and $T$ satisfies the {\rm(dGSF)} condition.   
\end{thm}

\begin{proof}
Let  $\Phi: \B(\D) \to L(X)$ be a bounded algebra homomorphism with $\Phi(1)= I$ and $\Phi(z) = T$.   The functions $z^k$ for $k\ge1$ all have $\B$-norm equal to 2, so the operators $T^k$ are bounded uniformly in $k$.   Let $w \in \D$.  Then $\rho_w(z) (1-{w}z) = 1$ for all $z\in\D$, and hence $\Phi(\rho_w) = (I-{w}T)^{-1}$.

 Let $h$ be a continuous function on $[-\pi,\pi]$ with $\|h\|_\infty = 1$, let $r \in (0,1)$, and let $G_{r,h}$ be defined by \eqref{grh}.    Then \eqref{B09} shows that
\[
\Phi(G_{r,h}) = \int_{-\pi}^\pi h(\theta) \Phi(\rho_{re^{-i\theta}}^2) \,d\theta = \int_{-\pi}^\pi h(\theta) (I - re^{-i\theta} T)^{-2} \,d\theta.
\]
Using \eqref{B11} we obtain, for $x \in X$ and $x^* \in X^*$ with $\|x\|=\|x^*\|=1$,
$$\begin{aligned}
(1-r) \left| \int_{-\pi}^\pi h(\theta)  \langle (I-re^{-i\theta}T)^{-2}x, x^* \rangle \,d\theta \right|
&= (1-r) \left| \langle \Phi(G_{r,h})x, x^* \rangle \right| \\ 
& \hspace{-20pt} \le (1-r) \|\Phi\| \, \|G_{r,h}\|_\B \le 6\pi \|\Phi\|.
\end{aligned}$$
It follows that
\[
(1-r)  \int_{-\pi}^\pi  \big| \langle  (I - re^{i\theta} T)^{-2}x, x^*  \rangle \big| \,d\theta  \le 6\pi  \|\Phi\|.
\]
This establishes \eqref{dgsf2}, and then \eqref{dgsf} follows from Proposition \ref{vdgsf}.
\end{proof}

The following result can be extracted from the results and proofs in \cite[Section 2]{CG08}, but it is not explicitly stated there.   It is analogous to \cite[Theorem~6.5]{BGT2}. 

\begin{prop} \label{isom}
Let $T$ be an invertible operator on a Banach space, such that $\sup_{n\in\Z} \|T^n\| < \infty$,  and assume that $T$ satisfies the {\rm(dGSF)} condition.   Then $T$ is polynomially bounded.
\end{prop}

\begin{proof}
By Proposition~\ref{vdgsf}, $T$ satisfies the condition \eqref{dgsf3}, which is the same condition as \cite[(2.5)]{CG08}. Hence $T$ satisfies condition~(2) in \cite[Theorem~2.2]{CG08}.  By following the proof in \cite[Lemma 2.3]{CG08} that condition~(3) in \cite[Theorem~2.2]{CG08} implies (2.6) and then (2.1) in \cite{CG08}, we establish that $T$ has a bounded $C(\T)$-calculus, and in particular $T$ is polynomially bounded.
\end{proof}

Since the assumptions in Proposition \ref{isom} include $\sup_{n\in\Z} \|T^n\| < \infty$, it is easy to see that polynomial boundedness of $T$ is equivalent to having a bounded $C(\T)$-calculus;  it is a simple special case of an argument in \cite[p.\,122]{CG08}.   Moreover, the same assumption removes the need to use \cite[Lemma~1.1(a)]{vC83} in the proof of \cite[Lemma 2.3]{CG08}. 

  In combination with Theorem \ref{cfbd}, Proposition \ref{isom} shows that 
   an invertible isometry which is not polynomially bounded does not have a bounded $\B$-calculus.   Such operators exist (see Example \ref{iinpb}), so this confirms that $W^+(\D)$ is properly included in $\B(\D)$.   An abstract existence result was proved in \cite[Theorem 4.1]{Zar05}, where it is shown that,  given a closed subset $E$ of $\T$, every isometry with spectrum contained in $E$ is polynomially bounded if and only if $E$ is a Helson set and satisfies spectral synthesis for~$W(\T)$. 
  
We now give a simple explicit example of an invertible isometry which is not polynomially bounded.   It is closely related to the example in \cite[Remark~2.14(2)]{Arn21}, but we present it differently.

\begin{exa} \label{iinpb}
Let $T$ be the bilateral shift on $W(\T)$ defined by
\[
(Tf)(z) = z f(z), \qquad f \in W(\T), \, z \in \T.
\]
Then $T$ is an invertible isometry on $W(\T)$.  Suppose that $T$ is polynomially bounded on $W(\T)$.  Since $p = p(T)1$ for each trigonometic polynomial $p$,
\[
\|p\|_\infty \le \|p\|_W = \|p(T)1\|_W \le C\|p\|_\infty,
\]
for some constant $C$.  However the norms $\|\cdot\|_W$ and $\|\cdot\|_\infty$ are not equivalent.  Thus $T$ is not polynomially bounded.
\end{exa}

The $\B$-calculus has standard spectral properties, as follows.

\begin{thm}  \label{SIT}
Let $T$ be a power-bounded operator on a Banach space $X$, satisfying the {\rm(dGSF)} condition. Let $f\in\B(\D)$.  Then $\sigma(f(T))=f(\sigma(T))$. 
Moreover the following statements hold for any $\l\in\C\colon$
\begin{enumerate} [\rm(a)]
\item  If $x \in X$ and $Tx = \l x$, then $f(T)x = f(\l)x$.
\item  If $x^* \in X^*$ and $T^*x^* = \l x^*$, then $f(T)^*x^* = f(\l)x^*$.
\item  If $(x_n)_{n\ge1}$ is a sequence of unit vectors in $X$ and $\lim_{n\to\infty} \|Tx_n - \l x_n\| = 0$, then $\lim_{n\to\infty} \|f(T)x_n - f(\l)x_n\| = 0$.
\end{enumerate}
\end{thm}

\begin{proof}
We begin by proving the spectral mapping property. We use the Banach algebra method in a similar way to \cite[Theorem~4.17]{BGT}.   Let $f\in\B(\D)$, with power series $f(z)=\sum_{n=0}^\infty a_nz^n$, and let $\A$ be a closed commutative subalgebra of $L(X)$ containing $I$, $T$, $(\l I-T)^{-1}$ for all $\l\in\rho(T)$ and $(\l I-f(T))^{-1}$ for all $\l \in \rho(f(T))$.  Then the spectra of $T$ and  $f(T)$ in $\mathcal{A}$ coincide with their spectra in $L(X)$.
Let  $\chi$ be any character of $\mathcal{A}$ and let $\l=\chi(T)$. Then $|\l|\le1$. Let $r \in (0,1)$ and $f_r(z):=f(rz)$ for $z\in\D$. Then $f_r\in W^+(\D)$ and hence $f_r(T)=\sum_{n=0}^\infty a_n r^n T^n = f(rT) $ by Corollary~\ref{frt}(b).   It follows that \begin{equation} \label{farr}
f_r(T) = f(0) + \frac{2}{\pi} \int_\D  \log \frac{1}{|z|}\,f'(z) rT (I-r\overline{z}T)^{-2} \, dA(z),
\end{equation}
where the integral is absolutely convergent in the operator norm. Note that $\chi((I-r\overline{z}T)^{-1}) = (1-r\overline{z}\l)^{-1}$ whenever $z \in \D$. 
 Hence, applying $\chi$ to~\eqref{farr} and using Proposition~\ref{repf}, we obtain
\[
\begin{aligned}
\chi(f_r(T)) 
&=  f(0) + \frac{2}{\pi} \int_\D  \log \frac{1}{|z|}\frac{f'(z) r\l }{(1-r\overline{z}\l)^{2}} \, dA(z)\\
& =f(0) + \frac{2}{\pi} \langle \rho_{r\l},f \rangle_{\B} =f(r\l).
\end{aligned}
\]
By Proposition~\ref{prp:B_prop}(e) and boundedness of the $\B$-calculus, $\lim_{r \to 1-} f_r(T) = f(T)$ in operator norm, so
\[
f(\chi(T))=f(\l) = \lim_{r \to 1-} f(r\l) = \lim_{r \to 1-} \chi(f_r(T)) = \chi(f(T)).
\]
Since $\chi$ was an arbitrary character, we conclude that $\sigma(f(T))=f(\sigma(T))$.

We now prove the remaining statements.

\noindent (a)  From (\ref{dBfc}) and Proposition \ref{repf}, for all $x^* \in X^*$, we have
\begin{align*}
\langle f(T)x,x^* \rangle &= f(0) \langle x,x^* \rangle +\frac{2}{\pi} \int_\D \log \frac{1}{|z|}  \frac{f'(z)\l}{(1 - \overline{z}\l)^2}  \langle x,x^* \rangle \,dA(z) \\
&= \Big(f(0)  + \frac{2}{\pi} \langle \rho_{\l},f \rangle_{\B} \Big)\langle x,x^* \rangle=f(\l) \langle x,x^* \rangle.
\end{align*}

\noindent (b)  This is similar to (a).

\noindent (c)  Let $Y = \ell^\infty(X)/c_0(X)$, where $\ell^\infty(X)$ is the Banach space of all bounded sequences $(x_n)_{n\ge1}$ in $X$, with the $\ell^\infty$-norm, and $c_0(X)$ is the subspace of sequences such that $\lim_{n\to\infty} \|x_n\|=0$.  By Theorem~\ref{dbfc}, the operator $T$ has a bounded $\B$-calculus on $X$ and it induces an operator $T_Y$ on $Y$ with a  bounded $\B$-calculus.  

Suppose that $\|x_n\|=1$ for all $n$ and $\lim_{n\to\infty} \|Tx_n-\l x_n\|=0$.  Let $\mathbf{x} = (x_n)_{n\ge1}$, and let $y = \pi(\mathbf{x}) \in Y$, where $\pi : \ell^\infty(X) \to Y$ is the quotient map.   Then $T_Y(y) = \l y$.  By applying (a) to $T_Y$, we see that $f(T_Y)y = \l y$, and this implies (c).
\end{proof}

If $T$ is a Ritt operator and $f\in\B(\D)$, there is a slightly simpler proof of the spectral mapping property $\sigma(f(T))=f(\sigma(T))$. Indeed, in this case we may apply any character $\chi$ of the algebra $\A$ considered in the above proof directly to the operator-valued formula~\eqref{ritt}, with no need for an approximation argument.

\section{Proof of Theorem \ref{KTB}}  \label{KT}

In this section, we prove Theorem \ref{KTB}.   For convenience, we repeat the statement here.  

\begin{thm} \label{KTB2}
Let $T$ be a power-bounded operator on a Banach space $X$, and assume that $T$ has a bounded $\B$-calculus.   Let $f \in \B(\D)$, and assume that $f$ vanishes on $\sigma_u(T)$.   Then $\lim_{n\to\infty} \|T^nf(T)\| = 0$.
\end{thm}

\begin{proof}
If $\T \subseteq \sigma(T)$, then the assumption on $f$ implies that $f$ vanishes on $\T$.  Then, by holomorphy, $f$ vanishes on $\D$, and $f(T)=0$.

Now assume that $\T$ is not contained in $\sigma(T)$.   We will first show that $\lim_{n\to\infty} \|T^nf(T)x\| = 0$ for all $x \in X$.  We start the argument by adopting a construction which has been used in several papers on the Katznelson--Tzafriri theorem and related topics, such as \cite{ESZ1}, \cite{KvN97}, \cite{Lek09};  see  \cite[Proposition 2.1]{Vu97} in particular.  By considering the seminorm
\[
q(x) := \limsup_{n\to\infty} \|T^nx\|, \qquad x \in X,
\]
taking the quotient space $X/\operatorname{Ker} q$ with the norm induced by $q$, and then taking the completion, one arrives at a Banach space $Z$, a bounded map $\pi : X \to Z$ with dense range such that $\|\pi(x)\|_Z = q(x)$ for $x \in X$, and an isometry  $V$ on $Z$ induced by $T$ so that $V\pi = \pi T$ and $\sigma(V) \subseteq \sigma(T)$.  Now $V$ has no approximate eigenvalues in $\D$ and $\sigma(V)$ is not equal to $\overline{\D}$, so $\sigma(V) \subseteq \T$.  Thus  $V$ is invertible and $\sigma(V) \subseteq \sigma_u(T)$.  In addition any operator $S \in L(X)$ which commutes with $T$ induces an operator $S_Z \in L(Z)$, and the map $S \mapsto S_Z$ is a contractive algebra homomorphism with $T_Z = V$. 

Let $g \in \B(\D)$.   Then $g(T)$ commutes with $T$, so there is an induced operator $g(T)_Z$ on $Z$.   The mapping $g \mapsto g(T)_Z$ is a bounded algebra homomorphism from $\B(\D)$ to $L(Z)$ mapping $1$ to $I$ and $z$ to $V$.   Thus $V$ has a bounded $\B$-calculus.   By Theorem \ref{cfbd}, $V$ satisfies the (dGSF) condition, so by Proposition~\ref{isom} the invertible isometry $V$ is polynomially bounded.   

Now we can apply Theorem \ref{KvNthm}  and we conclude that 
\[
\lim_{n\to \infty} \|V^nf(V)\| = 0.
\]
Since $V$ is an isometry this implies that $f(V)=0$ and hence $\pi f(T) = 0$.  In turn this implies that $\lim_{n\to\infty} \|T^nf(T)x\|=0$ for all $x \in X$.

Now we improve this result to show that $\lim_{n\to\infty} \|T^nf(T)\|=0$.  For this we use a minor variation of arguments previously used to strengthen the convergence in this way by passing from an operator on $X$ to an operator on the space $\ell^\infty(X)$.   
Suppose that the claim is not true.   Then there exist $\ep>0$ and a subsequence $(n_k)_{k\ge1}$ such that $\|T^{n_k}f(T)\| > \ep$ for all $k$.  Let $x_k \in X$ be such that $\|x_k\|=1$ and $\|T^{n_k}f(T)x_k\| > \ep$.   Now let $T_\infty$ be the operator on $\ell^\infty(X)$ defined by $T_\infty\mathbf{y} = (Ty_k)_{k\ge1}$, where $\mathbf{y} = (y_k)_{k\ge1}$, and let $\mathbf{x} = (x_k)_{k\ge1} \in \ell^\infty(X)$.   Then $T_\infty$ is power-bounded with a bounded $\B$-calculus, and $\sigma(T_\infty) = \sigma(T)$, so the conclusion of the previous paragraph implies that 
$\lim_{k\to\infty} \|T_\infty^{n_k}f(T_\infty)\mathbf{x}\|_\infty= 0.$
On the other hand,
\[
 \|T_\infty^{n_k}f(T_\infty)\mathbf{x}\|_\infty \ge \|T^{n_k}f(T)x_k\| > \ep,  \qquad k \ge1.
 \] 
 This contradiction completes the proof.
\end{proof}

\begin{rems}  \label{remend}
(a)  The condition that $f$ vanishes on $\sigma_u(T)$ is a necessary condition for the conclusion of Theorem~\ref{KTB} to hold.  Assume that  $f\in\B(\D)$ and $\lim_{n\to\infty} \|T^nf(T)\|=0$, and let $\l \in \sigma_u(T)$.  Then $\l^nf(\l) \in \sigma(T^nf(T))$ for all $n\in\Z_+$, by Theorem~\ref{SIT}.  Hence 
$|f(\l)| = |\l^n f(\l)| \le \|T^nf(T)\| \to 0$ as $n \to \infty,$  so $f$ vanishes on $\sigma_u(T)$.

\noindent
(b) Parts of the proof of Theorem \ref{KTB} above have some similarities with the proof of \cite[Theorem 2.10]{ESZ1}. This is not surprising as the contexts of the two theorems are quite similar, and this raises the possibility of finding a single result which covers both theorems.  Our proof of Theorem \ref{KTB} relies also on the (dGSF) condition being equivalent to the relevant operator having a bounded $\B$-calculus (which has its origins in \cite{BGT2}), and on results in \cite{KvN97} and \cite{CG08} in the last part of the proof.   None of these methods were available when the paper \cite{ESZ1} was published.

\noindent
(c)   In the final paragraph of the proof above, instead of using the operator $T_\infty$ on $\ell^\infty(X)$, we could have used the operator of left multiplication by $T$ on the Banach algebra $L(X)$.
The proof above shows that the (dGSF) condition for $T$ on $X$ implies the (dGSF) condition for both those operators.  It seems likely that this can be shown directly, without going through $\B$-calculus.

\noindent
(d)  A direct analogue of Theorem \ref{KTthm} for negative generators of $C_0$-semi\-groups was proved in  \cite[Th\'eor\`eme III.4]{ESZ2} and \cite[Theorem 3.2]{Vu92}, independently;  see also \cite[Theorem 5.7.4]{ABHN}.     We have recently proved an analogue of Theorem \ref{KTB} using broadly the same approach as in this paper, but the details were more involved and sufficiently different from the material presented in this paper to warrant a separate exposition; see \cite{BS22}.
\end{rems}

The main result in \cite[Theorem 1.2]{BS22} is formulated for some Banach algebras which are larger than $\B(\C_+)$. The corresponding extension of Theorem~\ref{KTB} is as follows.
 
\begin{thm} \label{KTB3}
Let $\mathcal{A}$ be a Banach algebra such that $\mathcal{B}(\mathbb{D})$ is continuously included in $\mathcal{A}$ and $\mathcal{A}$ is continuously included in $A(\mathbb{D)}$.   Let $T$ be a power-bounded operator on a Banach space, with a bounded $\mathcal{A}$-calculus.  Let $f \in \mathcal{A}$, and assume that $f$ vanishes on $\sigma_u(T)$.   Then $\lim_{n\to\infty} \|T^nf(T)\|=0$.
 \end{thm}
 
The proof of this is a minor variation of the proof of Theorem~\ref{KTB} given above.   Examples of Banach algebras satisfying the assumptions may be found in \cite{Pel82}.

\end{document}